\documentclass[a4paper,11pt,reqno]{amsart}
\usepackage[T1]{fontenc}
\usepackage[utf8]{inputenc}
\usepackage[margin=1.0in]{geometry}
\usepackage{lmodern}
\usepackage[english]{babel}
\usepackage{graphicx}
\usepackage{amsmath}
\usepackage{amsfonts}
\usepackage{amssymb} 
\usepackage{amsthm}
\usepackage{bbm}
\usepackage{bm} 
\usepackage{lipsum}
\usepackage{xcolor}
\usepackage{url}
\usepackage{eurosym}
\usepackage{hyperref}
\usepackage{subcaption}
\usepackage{comment}

\allowdisplaybreaks

\theoremstyle{plain}
\newtheorem{thm}{Theorem}
\newtheorem{lem}[thm]{Lemma}
\newtheorem{prop}[thm]{Proposition}

\newtheorem{rem}[thm]{Remark}
\theoremstyle{definition}

\def \be {\begin{equation}}
	\def \ee {\end{equation}}

\def\red{\textcolor{red}}

\renewcommand{\phi}{\varphi}
\renewcommand{\epsilon}{\varepsilon}
\renewcommand{\tilde}{\widetilde}
\renewcommand{\hat}{\widehat}
\renewcommand{\bar}{\overline}

\begin{document}

	\title{A mean field model for the development of renewable
		capacities}
	\author{Clémence Alasseur}
	\author{Matteo Basei}
	\author{Charles Bertucci}
	\author{Alekos Cecchin}

	\address[C. Alasseur]
	{\newline \indent EDF R\&D and FiME, Laboratoire de Finance des March\'es de l’\'energie
		\newline
		\indent 7 boulevard Gaspard Monge, 
		91120 Palaiseau,  France}
	\email{clemence.alasseur@edf.fr}

	\address[M. Basei]
	{\newline \indent EDF R\&D and FiME, Laboratoire de Finance des March\'es de l’\'energie
		\newline
		\indent 7 boulevard Gaspard Monge, 
		91120 Palaiseau,  France}
	\email{matteo.basei@edf.fr}

	\address[C. Bertucci]
	{\newline \indent Centre de Math\'ematiques Appliqu\'ees, UMR7641 \'Ecole Polytechnique
		\newline
		\indent Route de Saclay, 
		91128 Palaiseau, France}
	\email{charles.bertucci@polytechnique.edu}

	\address[A. Cecchin]
	{\newline \indent Dipartimento di Matematica ``Tullio Levi-Civita'', Università di Padova
		\newline
		\indent Via Trieste 63, 35121 Padova, Italy}
	\email{alekos.cecchin@unipd.it}

		\thanks{C. Alasseur acknowledges support from FIME and ANR ECOREES (project ANR-19-CE05-0042).
		C. Bertucci acknowledges a partial support from the FDD Chair. A. Cecchin benefited from the support of LABEX Louis Bachelier Finance and Sustainable	Growth (through project ANR-11-LABX-0019), ECOREES ANR Project, FDD Chair and Joint Research Initiative FiME}
	
	\date{May 20, 2023. Second version}

	\begin{abstract}
		We propose a model based on a large number of small competitive producers of renewable energies, to study the effect of subsidies on the aggregate level of capacity, taking into account a cannibalization effect. We  first derive a model to explain how long-time equilibrium can be reached on the market of production of renewable electricity and compare this equilibrium to the case of monopoly. Then we consider the case in which other capacities of production adjust to the production of renewable energies. The analysis is based on a master equation and we get explicit formulae for the long-time equilibria. We also provide new numerical methods to simulate the master equation and the evolution of the capacities.  Thus we find the optimal subsidies to be given by a central planner to the installation and the production in order to reach a desired equilibrium capacity.
	\end{abstract}
	
	\maketitle
	
	
	\section{Introduction}
	As electricity production is responsible for around a third of worldwide CO$_2$ emission, its decarbonization is of course a huge challenge. As such, the share of electricity produced from renewable capacities such as hydro, wind, photovoltaic is widely growing. However, to achieve decarbonization objectives, this share must still grow in a very large proportion. In 2021, the share of all renewable production had reached 30\%\footnote{https://www.iea.org/reports/global-energy-review-2021/renewables}. But many new investments are required in order to achieve a net-zero economy by 2050 as 90\% of global electricity generation should come from renewable sources at that point\footnote{https://www.iea.org/reports/net-zero-by-2050}.\\
	
	Electricity market is characterized by the constraint that production must be equal to the consumption at
	any time. In case of non respect of this constraint, the system can incur a power outage whose consequences might be highly problematic. For example, the total economic cost of the August 2003 blackout in the USA was estimated to be between seven and ten billion dollars \cite{UScouncil}. 
	As electricity can hardly be stored, hydro storage is limited in size, and developing a large fleet of batteries is still highly costly, the power production capacity must be high enough to cope with major peak load events, which can reach extreme levels compared to the average load. The consequence of such constraints on the production system is that some power plants are used rarely (only during extreme peak load) but remain necessary for the system security, and insuring their economical
	viability with energy markets only is not guaranteed. This question has already motivated a great amount
	of economic literature under the name of missing money \cite{Joskow}. That's also why some regulators took the lead on this matter through "strategic" reserves (see for example \cite{stratReserve}) or by the design of capacity remunerations (see for example \cite{Bubliz} for a review of theoretical studies and implementations of capacity remuneration mechanisms). \\
	
	Renewable energies have low variable costs so their introduction has made electricity prices lower, see for example \cite{brown2018capacity}  or\cite{Levin2015}. This effect is also sometimes referred as the cannibalisation effect (see \cite{cannibalisation}). This could reinforce the lack of incentives to invest in new capacities. To which level and at which rhythm new renewable capacity would develop in electricity markets is a key issue to regulators. Indeed, regulators can remove barriers by modifying the market design such as the rules of the markets, the level of competition or by providing financial incentives to the new capacities. The model we develop in this paper intends to provide explicit elements to guide regulators in such issues. \\
	
	We based our model on MFG (Mean Field Games) approach to represent numerous renewable producers who can decide to invest in new capacities. The use of MFG to analyse power system is not new. For example, it has been used to study how consumers can optimise their flexilibities against a dynamic price, see \cite{Couillet2012} for EV (Electrical Vehicle) charging, \cite{Paola2016},  \cite{Gomes2020} or \cite{Alasseur2020} for micro-storages or \cite{Paola2019} or \cite{Bauso} for TCL (Thermostatic Control Load). \\
	
	MFG are dynamic games involving an infinite number of small players. Such problems have been studied for a long time in Economics and a mathematical framework has been proposed in \cite{lasry2007mean,lions2007cours}. One of the key feature of the MFG theory is that it provides a tool to analyze the effect of an aggregate shock (or common noise) on the system, by means of the so-called master equation. When it is well posed, this equation, usually set on the set of probability measures \cite{cardaliaguet2019master}, contains all the information on the MFG. In several situations, the master equation can be posed on a finite dimensional space \cite{bertucci2021monotone},\cite{bitcoin} and it is thus easier to study. In this paper, we adopt a framework which is similar to the one of \cite{bitcoin}. In particular we shall study a master equation whose derivation relies on the expression of dynamical equilibria on markets, even if no precise MFG is introduced. This type of approach is quite natural in Economics like for example equations which result from no-arbitrage assumptions. Indeed, in those situations, it is not necessarily the mathematical problem of the arbitrageurs which is important but rather the equilibrium relations that it implies. \\
	
	In this paper, we analyse the long-term competition between producers who invest in renewable assets taking into account the cannibalisation effect. Moreover, we study the effect of subsidies on the aggregate level of capacity of production of renewable energies. We adopt a master equation formulation as in \cite{bitcoin}. For a given level of subsidies, we are able to explicitly characterise the level of capacities which would develop within the competition framework and compare it to the level of renewable capacities achieved in a monopoly setting. The equilibrium we obtain is explicit but also unique and our model provides a way to analyse the impact of the level of subvention. We prove that competition strengthens the development of renewable achieving a larger total renewable capacities and diminishes the profitability of producers compared to a non-regulated monopoly setting. Numerically, we can analyse the rhythm at which renewable develops and how fast a given renewable capacity target can be achieved. We are also able to compute the optimal levels of subsidies for a central planner wants to achieve a target of renewable capacity while saving the distributed subventions. In particular, we demonstrate that from a regulator point of view, to subsidize the cost of production as a decreasing amount of installed capacities is more efficient than keeping a fixed subvention over time. 
	We provide this study in two cases: one in which the renewable energies are the only adjustable capacities of production on the market and one in which the other capacities of production, the reserve, adjust to the production of renewable energies. When the strategic reserve adapts to the level of renewable capacity, we show in the paper the explicit and unique equilibrium of competitive and monopoly situation.

	Mathematically, this problem can be named as the one of controlling a MFG equilibrium. Even though this problem seems quite natural from the point of view of applications, it has received only few attention for the moment. This may be due to the intrinsic difficulty of the problem, that we shall not enter in in this paper. Here, we shall restrict ourselves to our particular problem and illustrate the results we obtain with numerical simulations.\\
	
	The rest of the paper is organized as follows. In Section \ref{sec:fixedreserve}, we introduce and study a simple model in which the capacity of electric production outside the renewable source is fixed. In this setting, we develop the case with a constant subvention to the production and also with a decreasing subvention with the total renewable capacity. In Section \ref{sec:extension}, we then proceed to study the more involved case in which this capacity evolves with the installed capacity of renewable.

	\section{Renewable capacity development with fixed reserve}
	\label{sec:fixedreserve}
	\subsection{A first general model} Assume that we have a multitude of producers of a renewable energy, that they are symmetric and non-atomic in the sense that every one of them is too small to have an influence on the market. Assume that the total (aggregate) available capacity of renewable energy $K_t$ (in MW) at time $t$ (in years) evolves according to
	\be\label{eq:196}
	\dot{K}_t = - \delta K_t + X_t.
	\ee
	The parameter $\delta >0$ is given and represents the decay of old installations due to time (in years$^{-1}$). The quantity $X_t$ is the variation of capacity, at time $t$, induced by the behaviour of the producers.
	Let $P_t$ be the market remuneration assimilated at the spot price in \euro/MWh and assume that it depends on the total (aggregate) available capacity $K_t$ in a decreasing way:
	\be 
	\label{PK}
	P_t = \frac{p}{K_t +\epsilon},
	\ee
	where $\epsilon$ is a fixed parameter in MW and $p$ a constant in \euro/h.
	This model for the spot price is directly inspired by structural approaches to model electricity prices, such as models developed in \cite{Barlow} or in \cite{Aid_struct}, which represent spot prices as the intersection of the production and demand curves in a stylised way. Because, in our model, the new capacities $X$ are renewable technologies with very low cost of production, large introduction of this type of technology mechanically decreases the spot price most of the time.\\
	
	Let $c_t$ be the cost of production, in \euro/MW, at time $t$.
	We introduce also a parameter $h$ (in hours/year) which accounts for the number of hours of production of the renewable energy, per year. Indeed, renewable technologies have the particularity to depend on meteorological conditions and as such do not run at full capacity all the time.\\
	
	We want to characterize the value of a single unit of production at time $t$, which we denote by $u_t$. Note that, equivalently, it represents the value of the "game" for a producer detaining a single unit of production at time $t$. From the previous assumptions, we can write $u_t$ as the sum of the expected payments given by a unit of production, before it defaults, which happens at rate $\delta > 0$. This yields (denoting by $a^+$ the positive part of $a$) 
	\be\label{eq:197}
	u_t = \mathbb{E}\left[\int_t^{\tau_t}e^{-r(s-t)}h\left(\frac{p}{K_s + \epsilon} -c_s\right)^+ ds\right].
	\ee
	The random variable $\tau_t$ is an exponential random variable of parameter $\delta^{-1}$ which models the time at which the unit of production will default.\\
	
	The strategic equilibrium which takes place in our model can be seen through the link between the equations \eqref{eq:197} and \eqref{eq:196}. Indeed to compute the value of a unit of production, we need to have anticipation on the evolution of the future total capacity installed. But on the other hand, the evolution of the total capacity installed should depend on the value of a unit of production: the more profitable the production of renewable, the more capacity should be installed. To make this link more apparent, we make the assumption that $X_t$ is a function of only the time $t$, the aggregate capacity of production $K_t$ and the value $u_t$ of a unit of production, and we write $X_t = F(t,K_t,u_t)$. We shall come back on this assumption later on.\\
	
	The previous assumption hints to seek $u_t$ as a function of only $t$ and $K_t$ which we note $u_t = U(t,K_t)$. This leads to a sort of dynamic programming formulation for $u_t$ which takes the form
	
	\begin{equation*} 
		U(t,K_t) = \mathbb{E}\left[\int_t^{t + dt \land \tau_t} e^{-r(s-t)}h\left(\frac{p}{K_s + \epsilon} - c_s\right)^+ + \mathbbm{1}_{\{t + dt< \tau_t\}} U(t + dt,K_{t + dt})\right],
	\end{equation*} 
	for $dt > 0$. Assuming that $c_t = c(t,K_t)$, simply computing $\tfrac{d}{dt}U(t,K_t)$,  applying the chain rule and using \eqref{eq:196} and the above equation, we find that $U$ should solve the PDE
	\be \label{met}
	\partial_t U(t,k) + (F(t,k,U(t,k))- \delta k)\partial_k U(t,k) - (r+\delta)U(t,k) + h \left(\frac{p}{k + \epsilon} - c(t,k)\right)^+ = 0
	\ee 
	on the domain $(t,k) \in (0,\infty)\times [0,\infty)$. Note that the previous equation is backward in time, that no boundary condition is needed at $t = \infty$ because of the presence of the discount factor $r$ and that no boundary condition is needed at $k = 0$ provided that $F(t,0,U(t,0))) \geq 0$. In the terminology of the MFG theory, this PDE is the master equation of the problem. \\
	
	Let us note that the well-posedness (existence and uniqueness) of the previous PDE yields the existence and uniqueness of an equilibrium in our model. Indeed, if we are able to compute $U$ through \eqref{met}, then we can also compute the evolution of $K_t$ through \eqref{eq:196}. {We stress that \eqref{met} is not the master equation of a mean field game, because $U$ does not represent the value of an agent in a Nash equilibrium, but the value of a unit. Instead, in our model, \eqref{met} is simply derived by the chain rule, and we call it master equation because it has the same mathematical structure of the mean field game master equation (in one space dimension), so that we easily get well-posedness results. It is in fact an equation for the value of  a unit  in equilibrium, but not a game, and the sole fact we use of MFG theory is the idea that producers are non-atomic (or small), as explained above.} \\
	
	
	\textbf{Equilibrium with a large  number of competitive producers.} 
	We now make two assumptions to make the precedent model more tractable and get, in particular, an explicit expressions for the equilibrium capacity. 	We assume first that the model is stationary, that is that no quantity depends explicitly on the time variable. This leads to the master equation
	\be 
	\label{mat:2}
	-(r + \delta)U(k) + (F(k,U(k))- \delta k)\partial_k U(k) + h \left(\frac{p}{k + \epsilon} - c(k)\right)^+ = 0
	\ee 
	which thus reduces to a singular ODE since only one variable is remaining.\\
	
	We now make the more crucial assumption that, for some constant $\lambda > 0$ and function $\alpha: \mathbb{R}_+ \to \mathbb{R}$, 
	\be \label{relationF}
	F(k,U) = \lambda(U - \alpha(k)).
	\ee 
	\begin{rem}
		This assumption has the important property that higher the value of a unit of production, the higher is the installation of new capacity and is justified by the following model.
	\end{rem}
	We now present a simple model with $N$ small competitive producers to justify the equations \eqref{relationF} and \eqref{mat:2} when $N$ is large. 
	Consider a producer whose capacity of production in (MW) at time $t$ is given by $S_t$.
	Let 
	\begin{itemize}
		\item $\alpha(k)$ be the cost of installation, in \euro/MW, given the installed (aggregate) capacity $k$
		\item $\beta_N$ be an adjustment cost, in \euro year/MW$^2$ which represents a friction to large volume of installation.
	\end{itemize}
	We shall come back later on why it is natural that $\beta$ depends on the number $N$ of producers.  
These costs depend on the type of the production technology. Later, we will consider the possibility for a central planner of subsidizing the cost of production $c$ and the cost of installation $\alpha$.
\\
We assume that the producer controls its (new) installed capacity $S_t$ through
\begin{equation*}  
	\dot{S}_t = -\delta S_t + y_t
\end{equation*}  
where $(y_t)_{t\geq 0}$ is its control, which represents the intensity of installation, in MW/year. Its reward is then 
\be 
\label{cost:ind}
\int_0^{\infty} e^{-rt} \Big( S_t h \left(\frac{p}{K_t + \epsilon}-c(K_t)\right)^+ - \alpha(K_t)  y_t - \beta_N y_t^2 \Big) dt.
\ee
which is in infinite horizon with a discount factor $r$ and, at every time $t$ (in years) given by the spot price minus the cost of production, multiplied by the capacity $S_t$ and the hours per year of production $h$, while it pays the installation cost, multiplied by the intensity of installation, and an adjustment cost. 

We assume, as above, that the producer neglects the impacts its production has on the aggregate production $K_t$, meaning that is is small compared to the total multitude of producers (as in mean field models). This implies, in particular, that $K_t$ is treated as a function of time only in the maximization of the above reward.
The maximization problem, in infinite horizon, is easily solved by means of Pontryagin's maximum principle: the Hamiltonian is 
\[
H(S,y, v)= (-\delta+ S)v +Sh \left(\frac{p}{K_t + \epsilon}-c(K_t)\right)^+ -\alpha(K_t)  y - \beta_N y^2, 
\]
where $v$ is the adjoint variable, which solves the ODE $\dot{v} = - \frac{\partial H}{\partial S} +rv$ and the optimal control maximized the Hamiltonian.  
Thus the  adjoint equation associated to this problem  is given by
\be 
\label{eq16}
\dot{v}_t = \delta v_t - h\left(\frac{p}{K_t + \epsilon}-c(K_t)\right)^+ + r  v_t 
\ee 
and the maximization leads to 
\be \label{eq:resultmax}
y_t = \frac{v_t - \alpha(K_t)}{2 \beta_N},
\ee 
which does not depend on $S_t$. From \eqref{eq16}, expressing $v_t= V(K_t)$ as a function of $K_t$, we then deduce (using the chain rule) that 
\begin{equation*}  
	-(r + \delta)V(K_t) + (\dot{K}_t)\partial_k V(K_t) + h\left(\frac{p}{K_t + \epsilon}-c(K_t) \right)^+ = 0.
\end{equation*} 
On the other hand, we can compute the evolution of the aggregate capacity $K_t$ as the sum of  the evolution of the capacity of all players. Let $\overline{S}_t$ be the new aggregate available capacity of the total multitude of producers, i.e. $\overline{S}_t =\sum_{i=1}^N S^i_t$, with $S_0=0$, thus the total available capacity is  
\be
\label{init:SK}
K_t= \overline{S}_t + k_0 e^{-\delta t},
\ee 
so that $K_0 =k_0$ is the initial capacity.  
This leads to 
\begin{align*}
	\dot{K}_t &= N y_t - \delta \sum_{i=1}^NS^i_t -\delta k_0 e^{-\delta t}\\
	& = N \frac{V(K_t) - \alpha(K_t)}{2\beta_N} - \delta K_t.
\end{align*}
Assuming that, in the limit $N \to \infty$, $\beta_N$ scales as $N\beta$ for some $\beta > 0$, we finally obtain that
\begin{equation*}  
	\dot{K}_t = \frac{V(K_t) - \alpha(K_t)}{2\beta} - \delta K_t.
\end{equation*}  
Hence, setting $\lambda = (2\beta)^{-1}$, we recover the required form $F(k,U) = \lambda(U - \alpha(k))$.
\begin{rem}
	Recall that in this $N$ players framework, $\beta_N$ is an adjustment cost. Hence it is natural that this cost grows with the number of producers (or at least of active producers) in the model. Indeed if $\beta_N$ is small, then the adjustment cost is also small and we can expect that a few big producers can adapt their production optimally quite easily. However, as $\beta_N$ grows, this adjustment becomes more and more important and the slowness of the producers to adjust allow rooms for smaller producers to install some small, non zero, capacity of production. Somehow this is what the relation \eqref{eq:resultmax} yields.
\end{rem}

\subsection {Constant subvention: competitive and monopoly equilibrium}
\label{subsec:const_sub}
We now assume that $c$ and $\alpha$ are \emph{constant} functions of $k$. These assumptions are also probably simplistic, however, they allow for several explicit computations which are helpful for some developments of this paper. We come back later, in \S \ref{sec:nonconstant}, on possible generalizations of the dependence of these costs on $k$.
This finally yields the master equation 
\be 
\label{master1}
-(r+\delta)U + \left(\lambda(U - \alpha) -\delta k\right) U' +  h \left(\frac{p}{k+\epsilon} -c\right)^+ =0,
\ee
and the total available capacity evolves according to 
\be 
\label{Kt}
\dot{K}_t = \lambda(U(K_t)-\alpha) - \delta K_t. 
\ee 
Note that the decay $\delta$ of  capacities due to time can also be interpreted such as in \cite{bitcoin} as the rate of technological progress of the renewable technology. This means that in this model the costs of production and of installation mechanically decrease through time as the technology improves.\\

The next results uses tools developed for mean field game theory and states well-posedness of the master equation as well as an important monotonicity property.

\begin{prop}
	\label{prop1}
	There exists a unique globally Lipschitz solution to \eqref{master1}. It is monotone decreasing and there exists a unique stationary state $k^*$ for the model. 
	
	Moreover, if either $\alpha \leq 0$ or $h c + \alpha(r+\delta) \leq \frac{p}{\epsilon}$, then $U(0) \geq \alpha$, which implies that $K_t\geq 0$ if $k_0 \geq 0$, and $K_t>0$ if $k_0>0$. 
\end{prop}

\begin{proof}
	We may rewrite Equation \eqref{master1} in the form 
	\[
	-rU - F(k, U)  \tfrac{d}{dk} U + G(k,U)=0,  
	\]
	where $F(k,U)= -\lambda(U-\alpha) +\delta k$ and 
	$g(k,U)= -\delta U +  h \big(\frac{p}{k+\epsilon} -c\big)^+$. Thanks to \cite{bertucci2021monotone, bitcoin}, existence and uniqueness of a Lipschitz and monotone decreasing solution hold true because of the monotonicity property of the couple $(F,G)$
	\[
	\langle (F,G)(u,k) - (F,G) (\tilde{u}, \tilde{k}), (u-\tilde{u}, k-\tilde{k}) \rangle \leq - \lambda |u-\tilde{u}|^2, \qquad \forall
	u, \tilde{u}, k,\tilde{k}, 
	\]
	which is easily verified in our case. Existence and uniqueness of the stationary state is a consequence of existence and uniqueness of a decreasing solution to the master equation.   
	
	To show that $U(0) \geq \alpha$, assume by contradiction that $U(0) < \alpha$. The equation for $\bar{U} = U - \alpha$ is 
	\[
	-(r+\delta) \bar{U} + (\lambda \bar{U} -\delta k) \bar{U}' + h \left(\frac{p}{k+\epsilon} -c\right)^+ -\alpha(r+\delta) =0 , 
	\]
	which, computed in $k=0$, gives 
	\be 
	-(r+\delta) \bar{U}(0) + \lambda \bar{U}(0) \bar{U}'(0) + h \left(\frac{p}{\epsilon} -c\right)^+ -\alpha(r+\delta) =0.
	\ee
	The first term is strictly positive ($-(r+\delta)\bar{U}(0) >0$), while the second is positive, since $U$ is decreasing ($\bar{U}(0) \bar{U}'(0) \geq 0$). Under the assumption, the last term is positive, 
	$h \left(\frac{p}{k+\epsilon} -c\right)^+ -\alpha(r+\delta) \geq 0$, whence the contradiction.
\end{proof}

\textbf{Stationary state and convergence of competitive market:} We here compute explicitly the equilibrium capacity for the above model given by \eqref{master1}-\eqref{Kt}. 
%
%
%
%
%
%
%
%
The stationary state is such that $\dot{K}=0$, that is, $\lambda(U(k^*)-\alpha) =0$. 
Inserting this in \eqref{master1}, we obtain the following system for the unknowns $(k^*, U(k^*))$:
\be 
\label{system:agents}
\begin{cases}
	\lambda(U(k^*)-\alpha) -\delta k^*=0, \\
	(r+\delta)U(k^*)= h \left(\frac{p}{k^*+\epsilon}-c \right)^+.
\end{cases}
\ee
We assume that $c\epsilon <p$. If $\frac{p}{k^*+\epsilon}-c \leq 0$, that is, $k^*\geq \frac{p-c\epsilon}{c}$ (if $c>0$), then we get $k^* = -\frac{\lambda \alpha}{\delta}$, which is not allowed, as $\alpha\geq 0$ and $k\geq 0$. If $c\leq0$, then the argument of the positive part is clearly positive. On the converse, if $c>0$ and $k^* < \frac{p-c\epsilon}{c}$ then the system  leads to the second order equation 
\be
\label{eq:2}
\delta k^2 + \left(\delta \epsilon +\alpha \lambda+ \frac{hc \lambda}{r+\delta} \right) k  +\alpha \lambda \epsilon +\frac{hc \lambda \epsilon}{r+\delta}- \frac{hp \lambda}{r+\delta} =0,
\ee
which admits a unique positive solution, if $\alpha  \epsilon < \frac{hp-hc\epsilon}{r+\delta}$, given by 
\[
k^*= \frac{ -\delta\epsilon-\lambda\alpha -\frac{hc \lambda}{r+\delta}+ \sqrt{\left(\delta\epsilon -\lambda\alpha - \frac{hc \lambda}{r+\delta}\right)^2 +4 \frac{hp \delta \lambda}{r+\delta}}}{2\delta} .
\]
Let us denote $\bar c = hc + (r+\delta)\alpha$ a quantity homogeneous to the cost in \euro.year$^{-1}$ of 1 MW of a new capacity. Indeed, 1 MW of new capacity installed at time $t=0$ leads to installation costs $\alpha$ plus the costs of production weighted by the decay factor of the capacity: $\int_0^{\infty} e^{-rt} \Big( e^{-\delta t} h c\Big)dt$, i.e. $h c. (r+\delta)^{-1}$ euros.
We see that $k^*$ is a \emph{decreasing} function of
\[
\bar{c} := hc+ (r+\delta) \alpha,
\]
given, if $\bar{c}\varepsilon <hp$, by 
\be
\label{kstar}
k^*(\bar{c})= \frac{ -\delta\epsilon-\frac{\bar{c} \lambda}{r+\delta}+ \sqrt{\left(\delta\epsilon  - \frac{\bar{c} \lambda}{r+\delta}\right)^2 +4 \frac{\delta \lambda h p}{r+\delta}}}{2\delta}.
\ee
It can be shown that $k^*(\bar{c}) < \frac{hp}{\bar{c}} -\varepsilon < 
\frac{p}{c} -\epsilon$, if $\alpha\geq 0$. It is interesting to analyse the sensitivity of $k^*$ with respect to the parameters of the problem. 
\begin{itemize}
	\item First, the higher the costs (of production $c$, as well as of construction $\alpha$), the less the capacity is developed, which is reasonable.
	\item In addition, $k^*$ increases with $p$, which means that the renewable capacity developed more when the spot market offers better rewards. 
	\item Finally, when the friction to installation $\beta$ is large, so $\lambda$ is small, no new renewable capacity is developed as the optimal capacity $k^*$ goes to zero. 
\end{itemize}
Therefore we assume that 
\begin{align}
	&\alpha\geq 0 \label{eq:admi_alpha}, \\
	&\bar{c} = hc+\alpha(r+\delta) < \frac{hp}{\epsilon}. \label{eq:admi_c}
\end{align}
These two assumptions are very natural. The first assumption directly comes from the fact that installing new capacities is costly and therefore $\alpha>0$. The second one makes the hypothesis that the global annual costs $\bar{c}$ are inferior to the remuneration a new capacity would receive from the spot market if no capacity were installed, which seems a necessary condition.  Indeed, this second assumption implies that if only one MW of renewable exists in the market, it should be able to found its remuneration on the spot market. 
Under \eqref{eq:admi_alpha} and \eqref{eq:admi_c}, we always have (if $K_0 < k^*$, so that $\dot{K}_t >0$)
\begin{equation*} 
	P_t -c = \frac{p}{K_t +\epsilon} -c > \frac{p}{k^*(\bar{c}) + \epsilon} -c >0,
\end{equation*}
so that we may remove the positive part. \\

We collect the above in the following result, where we also consider the problem of convergence of $K_t$ to $k^*$, as $t\rightarrow \infty$.

\begin{thm}
	\label{prop:exp-conv}
	If \eqref{eq:admi_alpha} and \eqref{eq:admi_c} hold, then the unique stationary equilibrium for the model \eqref{master1}-\eqref{Kt} is given by $k^*(\bar{c})$ in \eqref{kstar}.
	
	Let $K_0 < k^*$ and  $L$ be the Lipschitz constant of $U$ (we have 
	$L\leq \frac{2}{\epsilon^2 \delta}$ ). Then 
	$K_0 <K_t <k^*$ for any $t$ and 
	\be 
	(k^*-k_0) e^{-(\lambda L+\delta)t} \leq k^*-K_t\leq (k^*-k_0) e^{-\delta t}.
	\ee 
	In particular, $k^*$ is not reached in finite time. 
\end{thm}

\begin{proof}
	Since $U$ is decreasing, the drift $f(k) = \lambda(U(k)-\alpha) -\delta k$ is strictly decreasing. Thus $f(k) > f(k^*) = 0$ and then  $K_t$ is strictly increasing and strictly concave, so clearly $K_t \leq k^*$. 
	As $U$ is decreasing, we get 
	\begin{align*}
		\dot{K}_t &= \lambda(U(K_t) -\alpha) -\delta K_t \\
		&\geq \lambda (U(k^*) -\alpha ) - \delta k^* -\delta (K_t -k^*) \\
		&= -\delta (K_t -k^*).
	\end{align*}
	As $k^*$ is the stationary state, we get 
	\[
	\begin{cases}
		\frac{d}{dt} (k^* - K_t) \leq - \delta (k^* -K_t), \\
		(k^*-K_t)(0) = k^*-k_0,
	\end{cases}
	\]
	which implies $k^*-K_t\leq (k^*-k_0) e^{-\delta t}$. On the other hand, the Lipschitz continuity of $U$ yields 
	\begin{align*}
		\dot{K}_t &= \lambda(U(K_t) -\alpha) -\delta K_t \\
		&\leq  \lambda (U(k^*) + L| K_t- k^*| -\alpha ) - \delta k^* -\delta (K_t -k^*) \\
		&= -(\lambda L +\delta) (K_t -k^*).
	\end{align*}
	Therefore we obtain  $\frac{d}{dt} (k^* - K_t) \geq  -(\lambda L + \delta) (k^* -K_t)$, which gives 
	\[ k^*-K_t\geq  (k^*-k_0) e^{-(\lambda L + \delta) t}. \]
\end{proof}

\textbf{Monopoly:} we consider now a monopoly regime, in which there is a unique producer to develop this renewable capacities. 
Thus in particular $S_t = K_t$ above, and \eqref{PK} is assumed in the beginning (without the positive part), that is, the unique agent maximizes
\begin{equation*} 
	\int_0^{\infty} e^{-rt} \Big( K_t h\Big(\frac{p}{K_t +\epsilon}-c \Big) - \alpha  \dot{X}_t - \beta (\dot{X}_t)^2 \Big) dt.
\end{equation*} 
Note that this problem is similar to the one introduced to derive \eqref{master1}, except for the fact that here $N =1$ and there is only one producer. Then the adjoint equation becomes 
\[
\dot{u}_t = (r+\delta) u  - h\frac{p\epsilon}{(K_t+\epsilon)^2}+hc,
\]
which gives
\[
u_t = \int_t^{\infty} e^{-(r+\delta)(s-t)} h\Big(\frac{p\epsilon}{(K_t+\epsilon)^2}-c\Big) ds.
\]
This provides a different master equation:
\be 
\label{master2}
-(r+\delta)U + \left(\lambda( U - \alpha) -\delta k\right) U' +  \frac{hp \epsilon}{(k+\epsilon)^2} -hc =0.
\ee

\begin{prop}
	There exists a unique globally Lipschitz solution to \eqref{master2}. It is monotone de-
	creasing and there exists a unique stationary state $k^*_{mono}$ for the model. If $\bar{c}\varepsilon <hp$ then $k^*_{mono}$ is the unique positive root of the polynomial 
	\begin{equation} 
		\label{poly3}
		\delta k^3 + \left( 2\epsilon \delta +\frac{\lambda \bar{c}}{r+\delta}  \right) k^2 
		+\left( \delta \epsilon^2 +2\epsilon \frac{\lambda \bar{c}}{r+\delta} \right) k
		+ \frac{\epsilon \lambda}{r+\delta} ( \bar{c} \epsilon - h p).
	\end{equation}
	
\end{prop}

\begin{proof} 
	Existence and uniqueness of a Lipschitz and monotone solution is shown as in Prop. \ref{prop1}. This implies that there exists a unique stationary state $k^*$. 
	%
	%
	%
	To find the stationary state, we have to solve the system 
	\be 
	\label{system:mono}
	\begin{cases}
		\lambda(U(k^*)-\alpha) -\delta k^*=0, \\
		(r+\delta)U(k^*)= \frac{hp\epsilon}{(k^*+\epsilon)^2}-hc ,
	\end{cases}
	\ee
	which gives the equation 
	\[
	\delta k^3 + \left( 2\epsilon \delta +\frac{\lambda hc}{r+\delta} +\lambda \alpha \right) k^2 
	+\left( \delta \epsilon^2 +2\epsilon \left(\frac{\lambda hc}{r+\delta} +\lambda \alpha \right) \right) k
	+\epsilon^2 \left(\frac{\lambda hc}{r+\delta} +\lambda \alpha \right) 
	- \frac{h p \lambda \epsilon}{r+\delta}.
	\] 
	Thanks to Descartes' rule of signs, this admits a unique positive solution if $c\epsilon <p$ and
	$
	\alpha \epsilon < \frac{hp-hc\epsilon}{r+\delta};
	$ 
	equivalently, this writes $\bar{c}\varepsilon <hp$ and $k^*$ is a root of \eqref{poly3}.  
\end{proof}


We now compare the two regimes. The following result shows that competition enables to reach higher installation capacities compared to monopoly. This also shows that as expected competition drives lower profitability for producers because the value of a single unit of production is then lower under competition compared to monopoly. As a result, the spot price is also higher under the monopoly compared to competition.

\begin{thm}	
	\label{lem:4}
	For every admissible value of the parameters, we have 
	\begin{equation*} 
		k^*_{mono} < k^*_{agents} ,
	\end{equation*} 
	where $k^*_{mono}$ and $k^*_{agents}$ are the equilibrium capacities in the monopoly regime and in the competitive small agents regime, respectively. 
\end{thm}

\begin{proof} 
	The stationary states  $k^*_{agents}$ and $k^*_{mono}$ are the unique solutions of systems \eqref{system:agents} and \eqref{system:mono}, respectively. The first equation is the same, and we can write $U(k^*)$ as a strictly increasing function of $k^*$; while in the second equation $U(k^*)$ is written as a strictly decreasing function of $k^*$. Comparing the rhs of these two equations, we note that  
	\[
	\frac{p\epsilon}{(k+\epsilon)^2} < \frac{p}{k+\epsilon} \qquad 
	\mbox{ if } k>0,
	\] 
	which implies the claim. 
\end{proof}

\subsection{Optimize subventions to control the competitive equilibrium capacity}
\label{sec:constantsub_optimizesub}
We consider now the problem where a central planner subsidizes the new capacities in order to achieve some target of renewable capacities. The central planner can either provide subvention at the installation of the new capacity, providing $\alpha_{sub}$, or when the capacity is running, providing $c_{sub}$ to the producer. The central planner has to determine the optimal level of subvention it provides to achieve its target while constraining the amount of subsidies it provides. In the model $\alpha$ is simply replaced by $\alpha -\alpha_{sub}$,  and the same for $c$. Thus $\alpha$ is a cost of installation and $\alpha_{sub}$ the subvention to installation (in \euro/MW), while $c$ is a cost of production and $c_{sub}$ is the subvention to production. 

Let us denote by $\bar{c}_{sub} = h c_{sub} + \alpha_{sub} (r+\delta)$ the global annual subvention in (\euro/MW)/year provided by the central planner to a new capacity. In the model, the producer is indifferent whether this global subvention should be paid once at the installation ($c_{sub}=0$), or during the lifetime of the production ($\alpha_{sub}=0$) or a combination of both. \\

From the point of view of the central planner, we consider the problem of choosing the parameters $\alpha_{sub}$ and $c_{sub}$ in order to reach a desired value $\bar{k}$ for the stationary production $k^*$. The central planner then aims at minimizing the distance between $k^*$ and $\bar{k}$, while paying reasonable subsidies. 
%
	Looking at \eqref{cost:ind} and \eqref{init:SK}, the total subvention to installation provided by the central planner at time $t$ is $\alpha_{sub}\sum_i \dot{X}_{i,t}$ and the total subvention to production is $h c_{sub} \overline{S}_t$. We have $\dot{X}_t = \dot{S}_t+\delta S_t$ and $\overline{S}_t = K_t -k_0 e^{-\delta t}$. The last equality gives $\dot{\overline{S}}_t = \dot{K}_t +\delta k_0 e^{-\delta t}$. Therefore, $\sum_i \dot{X}_{i,t} = \dot{\bar{S}}_t+\delta \bar{S}_t = \dot{K}_t + \delta K_t$. Hence, the central planner aims at minimizing
	%
	\be 
	\label{cost25}
	J_{sub}(\alpha_{sub}, c_{sub}) = 
	\mu (k^*(\bar{c}_{sub} ) - \bar{k})^2 
	+ \alpha_{sub} \int_0^\infty e^{-rt} ( \dot{K}_t + \delta K_t ) dt
	+ h c_{sub}\int_0^\infty e^{-rt} (K_t -k_0 e^{-\delta t}) dt, 
	\ee 
	where $\mu$ is a factor, in \euro/MW$^2$.
Integration by parts gives 
\begin{align*}
	&\alpha_{sub} \int_0^\infty e^{-rt} ( \dot{K}_t + \delta K_t ) dt
	+ h c_{sub}\int_0^\infty e^{-rt} (K_t -k_0 e^{-\delta t}) dt \\
	& = -\alpha_{sub} k_0 
	+ (\alpha_{sub} r +\alpha_{sub} \delta + h c_{sub}) \int_0^\infty e^{-r t}  K_t dt -k_0 \frac{h c_{sub}}{r+\delta} \\
	& = \bar{c}_{sub}  \int_0^\infty e^{-r t}  K_t dt - k_0 \frac{\bar{c}_{sub}}{r+\delta}
\end{align*}
and therefore the cost \eqref{cost25} is equal to 
\be 
\label{cost20bar}
\bar{J}_{sub}(\bar{c}_{sub}) = \mu (k^*(\bar{c}_{sub} ) - \bar{k})^2  
+ \bar{c}_{sub}  \int_0^\infty e^{-r t}  K_t dt - k_0 \frac{\bar{c}_{sub}}{r+\delta},
\ee
which depends only on $\bar{c}_{sub}$.  In our model, the producer and the regulator are sensitive to the global amount of subvention and are indifferent whether is is paid at the construction or during the exploitation of the power plant.\\

Let's note that equation \eqref{cost20bar} can be also rewritten as 
\be 
\label{cost20bar_bis}
\bar{J}_{sub}(\bar{c}_{sub}) = \mu (k^*(\bar{c}_{sub} ) - \bar{k})^2  
+ \bar{c}_{sub}  \int_0^\infty e^{-r t}  \bar{S}_t dt,
\ee
where very clearly it appears that the objective of the central planner is a compromise between achieving the target capacity $\bar{k}$ and the total subvention it provides to the total new capacity $\bar{S}$.
Recall that we assume that $\alpha, \alpha_{sub}$ and $c,c_{sub}$ are constants. 
The admissible domain, given by \eqref{eq:admi_alpha}-\eqref{eq:admi_c},  is for $\alpha -\alpha_{sub} \geq 0$ and $(\bar{c} -\bar{c}_{sub}) \epsilon < hp$, that is 
\begin{align*}
	&\alpha_{sub} \leq \alpha, \\
	&\bar{c}_{sub} = hc_{sub} +\alpha_{sub} (r+\delta) > \bar{c}-\frac{hp}{\epsilon}.
\end{align*} 
This admissible domain implies that the subvention to the installation could not be superior to the installation costs, meaning that the producer cannot earn free money by only installing a new capacity without running it. Let's recall that $\bar{c}-\frac{hp}{\epsilon}<0$ by assumption \eqref{eq:admi_c}, which means that the all positive subsidies $\bar{c}_{sub}$ to production are admissible. Therefore the admissible domain covers a large enough domain for our applications.

Setting $\bar{U}= U-\alpha$, equation \eqref{master1} becomes 
\be 
\label{Ubar}
-(r+\delta) \bar{U} + (\lambda \bar{U} -\delta k) \bar{U}' + \frac{h p}{k+\epsilon} -(hc-hc_{sub})-(r+\delta ) (\alpha-\alpha_{sub}) =0.
\ee 
We can reformulate the problem to depend on
\[
\bar{c} = h c + \alpha(r+ \delta) , \qquad 
\bar{c}_{sub} = h c_{sub} + \alpha_{sub} (r+ \delta).
\]
\\
The master equation \eqref{Ubar}, as well as the dynamics for $K_t$ in \eqref{Kt}, depends only on $\bar{c}-\bar{c}_{sub}$, and thus also the stationary state, which is given by 
\be
\label{kstarsub}
k^*(\bar{c}_{sub})= \frac{ -\delta\epsilon-
	\frac{(\bar{c} - \bar{c}_{sub}) \lambda}{r+\delta}+ \sqrt{\left(\delta\epsilon  - \frac{(\bar{c} - \bar{c}_{sub}) \lambda}{r+\delta}\right)^2 +4 \frac{\delta \lambda h p}{r+\delta}}}{2\delta}.
\ee
We see that $k^*(\bar{c}_{sub})$ is an increasing function of $\bar{c}_{sub}$.
Hence the cost $\bar{J}$ in \eqref{cost20bar} depends only on 
$\bar{c}_{sub}$.
Under monopoly, using  Theorem \ref{lem:4}, we note that more subvention would be required to achieve the same target $\bar k$.\\

\textbf{Numerical simulations}
We solve numerically  first the master equation  \eqref{Ubar} and then the ODE \eqref{Kt} for a large enough initial value. We assume $k_0=30$ GW and repeat for any value of $\bar{c}$ constant. There are two ways to solve \eqref{Ubar}: either we discretise the HJ equation
\begin{equation*}
	-r V +  \Big( -\delta k V +\lambda \frac{V^2}{2} +\log(k+\epsilon) -\bar{c} k\Big) =0,
\end{equation*}
solve it by the Newton method and then take the discrete derivative, or we solve directly \eqref{Ubar}, basically by the Euler implicit scheme backward, giving the terminal value at $k^*$; the result is approximately the same. More precisely, given the value $U(k^*)$, we discretise the interval $[0,k^*]$ and the value $u_i \approx U(k_i)$, knowing $u_{i+1}$ is given implicitly by 
\begin{equation*} 
	-(r+\delta) u_{i} 
	+ (\lambda( u_{i} -\alpha) -\delta k_i)
	\frac{u_{i+1}-u_{i}}{\Delta k}  
	+ \bigg( \frac{p}{k_i +\epsilon} -c\bigg)_+ =0.
\end{equation*}
Such equation is solvable for $u_i$ because $U$ is decreasing and, as a consequence,  $\lambda( U(k) -\alpha) -\delta k) >0$ for $0\leq k < k^*$. In fact, $u_i$ is given by 
\begin{equation*} 
	u_i = \frac{- \Delta k (r+ \delta) +\lambda u_{i+1} +\delta k_i + 
		\sqrt{(-\Delta k (r+ \delta) +\lambda u_{i+1} +\delta k_i)^2 
			+4\lambda  \bigg[ \Delta k \bigg( \frac{p}{k_i +\epsilon} -\bar{c} \bigg) -\delta k_i u_{i+1} \bigg] 
	}}{ 2\lambda}.
\end{equation*}


We obtain the following results fixing the parameters 
\begin{align*}
	&r=0.1 \mbox{year}^{-1}, \qquad \delta=\ln(2)/10 \mbox{year}^{-1},  \qquad \lambda=5 \frac{MW^2}{\mbox{\euro  year}} , \qquad \epsilon=0.1 MW, \qquad h = 3000 \mbox{hours.year}^{-1}\\
	&\alpha=1400 \mbox{\euro}/kW , \qquad c=15 \mbox{\euro}/ MW \mbox{h}, \qquad p = 6.5\cdot 10^6 \mbox{ \euro/ h}, \qquad k_0= 30 GW, 
	\qquad \bar{k} = 60 GW
\end{align*}
where $r, \lambda$ are taken from \cite{ABP}, as well as $\alpha$ and $c$, while $\delta$ comes from \cite{AD} and means that that the plant
loses 50\% of its value over 10 years.








We fix a reserve level at $Y_0 = 70$ GW which enters in the spot price as 
$\frac{p}{K_t + Y_0 + \epsilon}$; in fact, this corresponds to a new value for $\epsilon$ which becomes $Y_0 + \epsilon$ in previous calculations. 
The quantity $\bar{c}$ is $\bar{c} = hc + \alpha(r+\delta)= 282,040$ \euro/(MW year) and we fix $\mu= 1000$ \euro/MW. \\

The power market at time $t=0$ is therefore a system with 30 GW existing renewable capacities, 70 GW of reserves such as thermal plants which complement the renewable production. The central planner aims at doubling the renewable capacities to 60 GW. $Y_0$ and $k_0$ are directly inspired by French power system\footnote{The French TSO, Transport Service Operator, publishes the installed capacities for France in https://www.services-rte.com/fr/visualisez-les-donnees-publiees-par-rte/capacite-installee-de-production.html} and $p$ is set such that the spot price is around 60 or 50 \euro/MWh which is a common price level observed in the spot market in 2021. \\

The value of $\bar{c}_{sub}$ for which $k^*(\bar{c}_{sub}) = \bar{k}=60$ GW is 133,400 \euro/(MW year) which implies that the central planner would have to subsidize the producers around half of their cost to reach the target of 60GW of renewable in the system. 
The minimum of the objective function of the central planner $J$ corresponds to a slighly lower value of subvention $\hat{\bar{c}}_{sub} = 132,500$ \euro/(MW year). The corresponding equilibrium value is $k^*(\hat{\bar{c}}_{sub}) =59.2$ GW, which is 800 MW less then the target $\bar{k}=60$ GW. 
The difference between the target and the capacity achieved at the terminal time $T= 50$ years is $k^*(\hat{\bar{c}}_{sub})- K_{50} =2.4$ MW. 
The results are shown in figures \ref{fig:fixe_reserve_kstar} and \ref{fig:fixe_reserve_ktPt}. Figure  \ref{fig:fixe_reserve_ktPt} shows the time-dynamics of the installed capacity: the installed capacities gain 20 GW in only tree years and double in less than 10 years. Changing the market design and enabling a larger remuneration through the spot market (i.e. by increasing $p$) for renewable capacities enables to speed the rhythm at which the new renewable capacities are developed and also reduced the amount of subsidies. This is illustrated in figure \ref{fig:fixe_reserve_ktPt_1.5p}. 





\begin{figure}
	\centering
	\includegraphics[scale=0.3]{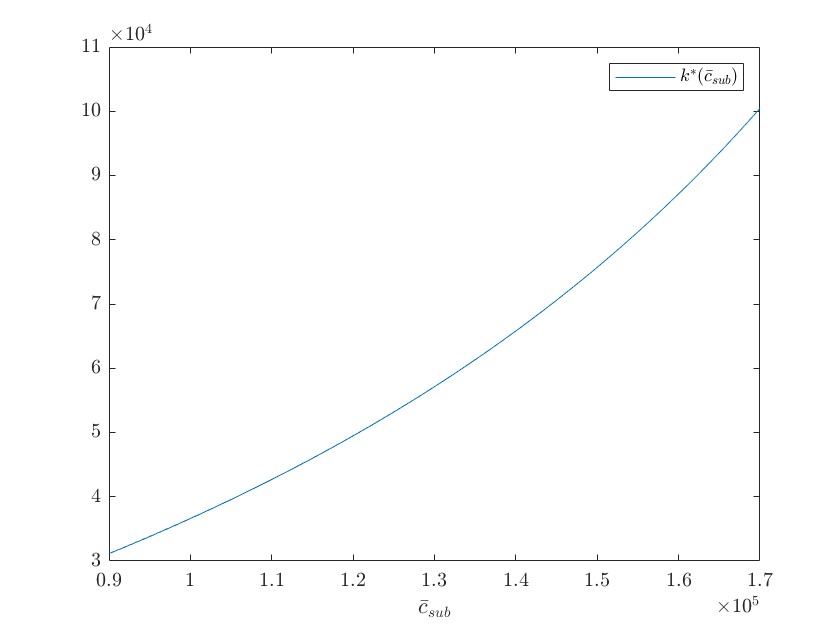}
	\caption{$k^\star$ in MW with respect to $\bar{c}_{sub}$ in \euro/MW/year}
	\label{fig:fixe_reserve_kstar} 
\end{figure}

\begin{figure}
	
	\centering
	\begin{subfigure}{.5\textwidth}
		\centering
		\includegraphics[scale=0.3]{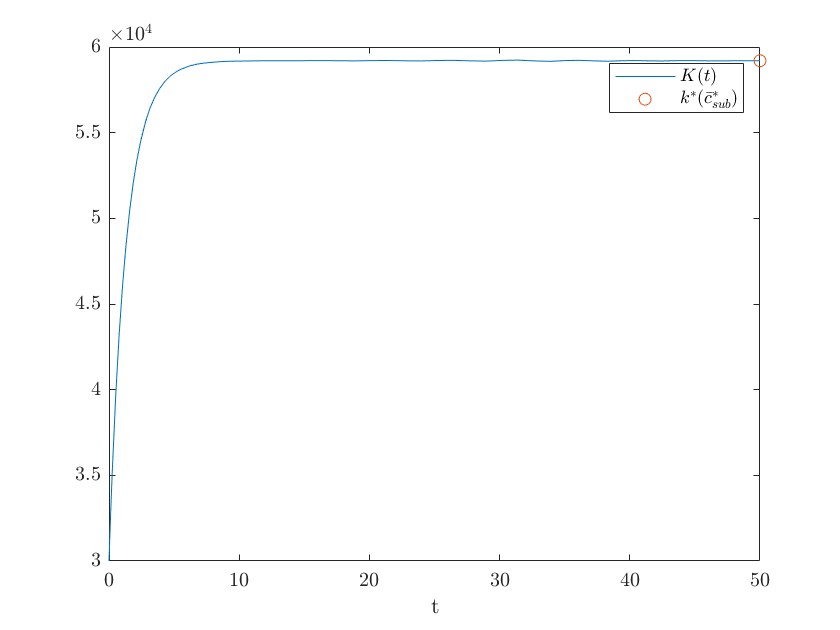}
	\end{subfigure}%
	\begin{subfigure}{.5\textwidth}
		\centering
		\includegraphics[scale=0.3]{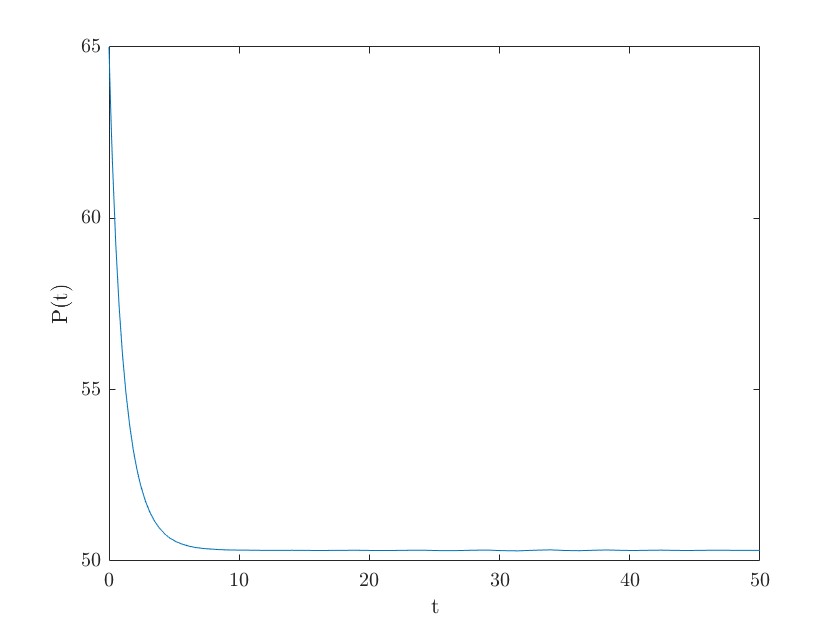}
	\end{subfigure}
	
	\caption{Time evolution of the installed capacity $k_t$ in MW (left) and of spot price $P_t$ in \euro/MWh (right)}
	\label{fig:fixe_reserve_ktPt}
\end{figure}

\begin{figure}
	\centering
	\begin{subfigure}{.5\textwidth}
		\centering
		\includegraphics[scale=0.3]{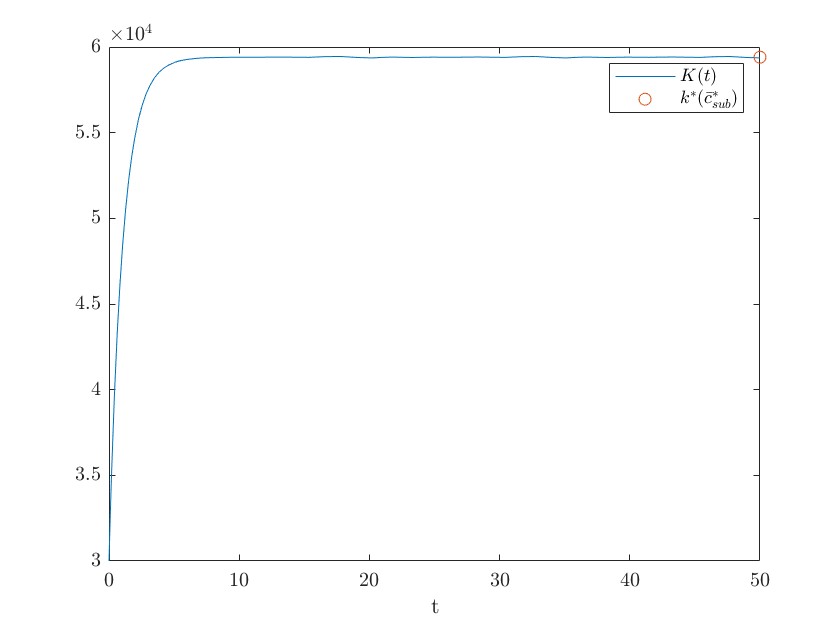}
	\end{subfigure}%
	\begin{subfigure}{.5\textwidth}
		\centering
		\includegraphics[scale=0.3]{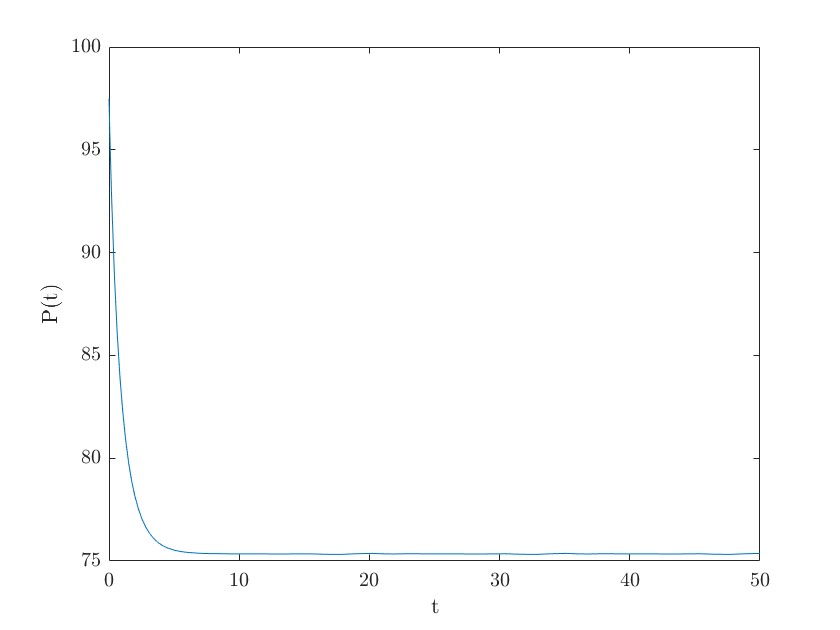}
	\end{subfigure}
	
	\caption{Time evolution of the installed capacity $k_t$ in MW (left) and of spot price $P_t$ in \euro/MWh (right) with higher spot remuneration, ie with $p=1.5 * 6.5.10^6$}
	\label{fig:fixe_reserve_ktPt_1.5p}
\end{figure}

\subsection{Non constant cost of production and subsidies}
\label{sec:nonconstant} 

In the previous section we considered the case of constant cost of production $c$, subsidy to production $c_{sub}$, cost of installation $\alpha$, subsidy to installation $\alpha_{sub}$ but decreasing along time with respect to the rate of technological progress $\delta$. 
We consider here the more general case of cost and subvention to production which depend on the total capacity in a decreasing way: the higher the capacity, the lower the cost of production and the subsidy. The cost and subvention to installation are still assumed to be constant with respect to the total capacity.
The most general case we can consider the case in which the cost of production is a function $\varphi(k)$ and the subsidy to production is a function $\psi(k)$. Thus the dynamics of the aggregate capacity, in the competitive regime, is still given by \eqref{Kt}, while the master equation \eqref{master1} becomes 
\be 
\label{master1bis}
-(r+\delta)U + \left(\lambda( U - \alpha) -\delta k\right) U' +  h \left(\frac{p}{k+\epsilon} - \phi(k) +\psi(k) \right)^+ =0,
\ee

\begin{thm}
	If the function $\frac{p}{k+\epsilon} - \phi(k) +\psi(k)$ is decreasing in $k$, then the master equation \eqref{master1bis} admits a unique Lipschitz solution, it is monotone, and there exists a unique equilibrium $k^*$, which is such that   
	\be 
	\label{equi:2}
	\left(\frac{\delta}{\lambda} k^* +\alpha \right)(r+\delta) =  h \left(\frac{p}{k^*+\epsilon} - \phi(k^*) +\psi(k^*) \right)^+.
	\ee 
	Moreover, there is exponential convergence to the equilibrium. 
\end{thm}
\begin{proof}
	Well-posedness is again a consequence of \cite{lions2007cours}, where existence and uniqueness is proved in the monotone case. Imposing $\dot{K}_t=0$ in the equilibrium, we find $\lambda( U(k^*) - \alpha) -\delta k^*=0$ and, putting this in the master equation \eqref{master1bis} we get $(r+\delta)U(k^*)= h \left(\frac{p}{k^*+\epsilon} - \phi(k^*) +\psi(k^*) \right)^+$. Thus we obtain \eqref{equi:2}, which is uniquely solvable since the left hand side is a strictly increasing function of $k^*$, while the left hand side is decreasing by assumption. Exponential convergence is proved as in Proposition \ref{prop:exp-conv}.  
\end{proof} 

In order to get explicit results, we consider specific shapes of the cost of production $\varphi(k)$ and the subvention to production $\psi(k)$: 
\[
\varphi(k) =c, \qquad \psi(k) = \frac{c^1_{sub}}{k+\epsilon} +c^2_{sub}.
\] 
This two functions correspond to constant cost of production  and to decreasing subvention to production with respect to the capacity. The central planner has then to choose the two parameters $c^1_{sub}$ and $c^2_{sub}$ in order to reach the desired equilibrium capacity. The subvention to production is not constant here, but depends on the total capacity $K_t$ in a decreasing way, if we impose $c^1_{sub} >0$. To keep the model tractable, we assume that it depends ok $K_t$ in the same way as the spot price. This expression is easy to understand for a producer as it means that the regulator proposes to subsidize production with a constant part and with a variable part which is proportional to the spot price. 
In fact, with respect to the first model given by \eqref{master1}, $p$ is replaced by $p+c^1_{sub}$ and $c$ by $c-c^2_{sub}$. Thus we obtain the constraints 
\begin{align*}
	\alpha_{sub}&\leq \alpha, \qquad c^1_{sub} >0\\ 
	\bar{c}^2_{sub} &= hc^2_{sub} + \alpha_{sub} (r+\delta) > \bar{c} - \frac{h}{\epsilon} (p+c^1_{sub}) 
\end{align*}
(where again $\bar{c}=hc  + \alpha (r+\delta)$) and the equilibrium capacity is given by 
\be
\label{kstarsub2}
k^*(c^1_{sub}, \bar{c}^2_{sub})= \frac{ -\delta\epsilon-
	\frac{(\bar{c} - \bar{c}^2_{sub}) \lambda}{r+\delta}+ \sqrt{\left(\delta\epsilon  - \frac{(\bar{c} - \bar{c}^2_{sub}) \lambda}{r+\delta}\right)^2 +4 \frac{\delta \lambda h (p+c^1_{sub})}{r+\delta}}}{2\delta}.
\ee
It is now a function of two variables and we see that it is an increasing function of both $c^1_{sub}$ and $\bar{c}^2_{sub}$. This implies that, to reach a desired equilibrium $\bar{k}$, the larger $c^1_{sub}$ can be chosen, the lower $\bar{c}^2_{sub}$, and vice versa. 

The central planner aims at reaching a desired equilibrium capacity $\bar{k}$. She can subsidize either the quantity $\bar{c}^2_{sub}$, as in the case of constant subsidy, or the quantity $c^1_{sub}$. Since the equilibrium capacity satisfies \eqref{eq:2} (with $p$ replaced by $p +c^1_{sub}$ and $c$ by $c-c^2_{sub}$ therein), we get that the optimal parameters $\bar{c}^2_{sub}$ and $c^1_{sub}$ are connected by the linear combination: 
\[
h(c^1_{sub}+p ) = (c-\bar{c}^2_{sub})(\epsilon+\bar{k})+ \frac{r+\delta}{\lambda}\delta\bar{k} ( \bar{k}+\epsilon).
\]
Similarly to \eqref{cost25}, the total cost of subvention paid by the central planner is then 
\[
J_{sub}(\alpha_{sub}, c^1_{sub}, c^2_{sub}) = \alpha_{sub} \int_0^\infty e^{-rt} ( \dot{K}_t + \delta K_t ) dt
+ h \int_0^\infty e^{-rt} \Big( c^2_{sub} +\frac{c^1_{sub}}{K_t +\epsilon} \Big) (K_t -k_0 e^{-\delta t}) dt, 
\]
which, integrating bu parts, becomes 
\[
\bar{J}_{sub}(\bar{c}^2_{sub}, c^1_{sub}) =  \bar{c}^2_{sub}  \int_0^\infty e^{-r t}  K_t dt - k_0 \frac{\bar{c}^2_{sub}}{r+\delta} 
+h c^1_{sub} \int_0^\infty e^{-rt}  \frac{K_t -k_0 e^{-\delta t}}{K_t +\epsilon} dt. 
\]

\bigskip

We observe that the function $\frac{K_t -k_0 e^{-\delta t}}{K_t +\epsilon}$ is less than 1, and thus much smaller than the quantity $K_t$ in the first integral. This implies that  the central planner is always better off to subsidize $c^1_{sub}$ with respect to $\bar{c}^2_{sub}$. Therefore, in this model, the optimal choice is to take $\bar{c}^2_{sub}$ as small as possible and for the regulator to subsidize production conversely with the installed renewable capacity. 


Fixing the parameters as in the previous subsection, we obtain that, if the subsidy $\bar{c}^2_{sub}=0$, then the optimal $c^{1,*}_{sub}$ to reach the desired equilibrium is $5.78 \cdot 10^6$ \euro/h, which is a fraction 0.89 of the spot price $p$. As observed above, with a larger $p$, the rhythm at which the renewable capacity reaches the equilibrium increases, i.e. $K_t$ gets close to the equilibrium $\bar{k}=60 GW$ in a shorter time.

\begin{figure}
	\centering
	\includegraphics[scale=0.3]{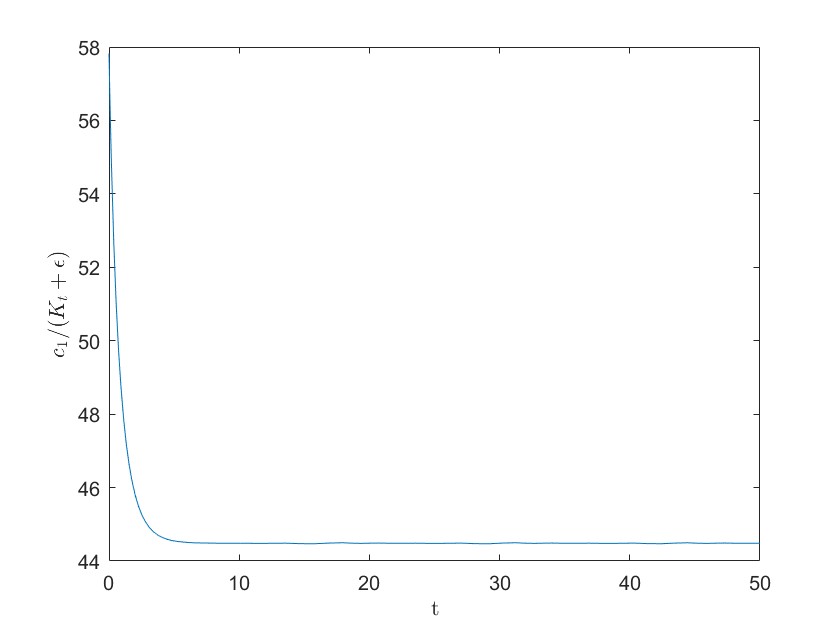}
	
	\caption{Time evolution of the optimal subsidy $\frac{c^{1,*}_{sub}}{K^*_t+\epsilon}$, in \euro/MWh.}
\end{figure}

\section{Renewable capacity development with adaptating capacity of reserves}
\label{sec:extension}
We introduce in the model another variable $Y_t$ which models the reserve. 
As before, $K_t$ is the available capacity of renewable energies and $(K_t, Y_t)$ evolve according to
\be 
\begin{split}
	\dot{K}_t &= \lambda( U(K_t, Y_t ) -\alpha) - \delta K_t, \\ 
	\dot{Y}_t &= f(K_t, Y_t). 
\end{split}
\ee 
The dynamics of the reserve is thus a deterministic function of the capacity $K_t$, given by the drift $f$: a central planner  automatically adapts the level of reserve to the installed capacities of renewable. This is perfectly known by renewable producer. Again, we assume that the spot price is in the form 
\begin{equation*} 
	P_t = \frac{p}{K_t +Y_t +\epsilon},
\end{equation*} 
and we consider here the case of constant costs of production and of installation. 
Thus, in line with the arguments of the previous section, we get that  $U(k,y)$ solves the PDE
\be 
\label{master:Y}
-(r+\delta) U + (\lambda (U-\alpha) -\delta k) \partial_k U + f(k,y) \partial_y U + h\left( \frac{p}{k+y +\epsilon} -c\right)^++ =0 .
\ee
and $U(K_t, Y_t)$ is given by
\begin{equation*} 
	U(K_t,Y_t) = \int_t^\infty e^{-(r+\delta)(s-t)} h\left(\frac{p}{K_s +Y_s  +\epsilon} -c\right)^+ds.
\end{equation*} 
Equation \eqref{master:Y} is a nonlinear PDE in space dimension 2. Because of the nonlinear term in the drift and of the non-divergence form, existence of classical solutions is not directly covered by existing theories. We provide below a proof of uniqueness which exploits the monotone structure and is inspired by the proofs of \cite{bertucci2021monotone, lions2007cours}. The questions of existence of classical (or Lipschitz) solutions seems a challenging mathematical question, beyond the scope of this paper.  
In fact, it might not be true for any choice of regular $f$, but may instead require a monotone structure on $f$, or on the drift $(\lambda(U-\alpha)-\delta k, f(k,y))$. The numerical study we give below in \S \ref{sec:numerics:2} suggests that a classical solution exists for a linear choice of $f$. 
\begin{prop}
	If $f$ is $C^1$, then there is at most one classical solution to \eqref{master:Y}, and it is decreasing in $k$.
\end{prop}

\begin{proof}
	Let $U^1$ and $U^2$ be two solutions and consider
	\[
	W(k_1, k_2, y_1, y_2) = ( U^1(k_1, y_1)-U^2(k_2, y_2) ) (k_1-k_2).
	\]
	We have 
	\begin{align*}
		\partial_{k_1} W &= \partial_{k_1} U^1(k_1, y_1) (k_1-k_2) +U^1 (k_1, y_1) -U^2(k_2, y_2), \\
		\partial_{k_2} W &= - \partial_{k_2} U^2(k_2, y_2) (k_1-k_2) - (U^1 (k_1, y_1) -U^2(k_2, y_2) ),  \\
		\partial_{y_1} W &= \partial_{y_1} U^1(k_1, y_1) (k_1-k_2) , \\
		\partial_{y_2} W &= - \partial_{y_2} U^2(k_1, y_1) (k_1-k_2)
	\end{align*} 
	and 
	\begin{align*}
		(r+\delta) W &= (r+\delta) U^1 (k_1-k_2) - (r+\delta) U^2 (k_1-k_2) \\
		&= \bigg[ ( \lambda(U^1-\alpha) -\delta k_1) \partial_{k_1} U^1
		-  ( \lambda(U^2-\alpha) -\delta k_2) \partial_{k_2} U^2 
		+ f(k_1, y_1) \partial_{y_1} U^1 - f(k_2, y_2) \partial_{y_2} U^2 \\
		& \quad  +h\Big( \frac{p}{k_1+ y_1 +\epsilon} -c\Big)_+ - 
		h\Big( \frac{p}{k_2+ y_2 +\epsilon} -c\Big)_+ \bigg] (k_1-k_2) \\
		&= ( \lambda(U^1-\alpha) -\delta k_1) \partial_{k_1} W
		+ ( \lambda(U^2-\alpha) -\delta k_2) \partial_{k_2} W \\
		&\quad - ( \lambda(U^1-\alpha) -\delta k_1) (U^1-U^2) 
		+ ( \lambda(U^2-\alpha) -\delta k_2) (U^1-U^2) \\
		&\quad + f(k_1, y_1) \partial_{y_1} W + f(k_2, y_2) \partial_{y_2} W \\
		& \quad +\bigg[ h\Big( \frac{p}{k_1+ y_1 +\epsilon} -c\Big)_+ - 
		h\Big( \frac{p}{k_2+ y_2 +\epsilon} -c\Big)_+ \bigg] (k_1-k_2) \\
		&\leq ( \lambda(U^1-\alpha) -\delta k_1) \partial_{k_1} W
		+ ( \lambda(U^2-\alpha) -\delta k_2) \partial_{k_2} W \\
		& \quad - \lambda(U^1-U^2)^2 + \delta W  
		+ f(k_1, y_1) \partial_{y_1} W + f(k_2, y_2) \partial_{y_2} W ,
	\end{align*}
	which yields 
	\[
	-r W + ( \lambda(U^1-\alpha) -\delta k_1) \partial_{k_1} W
	+ ( \lambda(U^2-\alpha) -\delta k_2) \partial_{k_2} W + f(k_1, y_1) \partial_{y_1} W + f(k_2, y_2) \partial_{y_2} W \leq 0. 
	\]
	Hence a maximum principle result gives $W\leq 0$. This gives the claim because, if there was a point $(k_1,y_1)$ where $U_1(k_1, y_1) \neq U_2(k_1, y_1)$, then (by continuity) there would be a point $(k_2,y_2)$ in a neighborhood for which  $W(k_1, k_2, y_1, y_2) >0$.  
\end{proof}

\subsection{Stationary state}
As in the previous section, we want to analyze the model's stationary states. As those states obviously depend on the nature of the function $f$, we fix here a form for $f$ and show how such a study can be made in this case.\\

We remark that $U(k,y)\geq 0$ because of the positive part. Thus the vector field for $k$, $\lambda (U(k,y)-\alpha) -\delta k$, is positive in $k=0$ if $\alpha \leq 0$. While if $\alpha >0$, we have to ensure that $U(0, y) \geq \alpha$ for any $y$. This holds true if $f(0,y) \leq 0$ and $U$ is decreasing in $y$ (or the opposite), and 
\[
\frac{hp}{y_{max} + \epsilon} -hc- \alpha(r+ \delta) \geq 0.
\]
Recalling that $\bar{c} = hc + \alpha (r+ \delta)$, it gives $y_{max} \leq \frac{hp}{\bar{c}} -\epsilon$. 
From the identity
\[ 
U(k,y) = \int_0^\infty e^{-(r+\delta)t} h\left(\frac{p}{K_t +Y_t +\epsilon} -c\right)_+ dt,
\]
and the fact $Y_t, K_t \geq 0$, we get $U(k,y) \leq \frac{h}{r+\delta}\left(\frac{p}{\epsilon} -c\right)$ and thus 
\[
\lambda(U-\alpha) -\delta k \leq 0  \quad \mbox{ if } 
k \geq \frac{\lambda}{\delta} \left[\frac{h}{r+\delta}\left(\frac{p}{\epsilon} -c\right) -\alpha\right] = k_{max}.
\]

The stationary state $(k^*, y^*)$ is such that the vector field is zero and so, evaluating the master equation on that point and calling $u^* =U(k^*, y^*)$,  we obtain the system 
\be 
\label{eq:37}
\begin{cases}
	\lambda( u^* -\alpha ) -\delta k^*=0, &\\
	f(k^*, y^*)=0, &\\
	-(r+\delta) u^* + h\left(\frac{p}{k^* +y^* +\epsilon} -c\right)_+ = 0. &
\end{cases}
\ee 
Observe that, if $\frac{p}{k^* +y^* +\epsilon} -c \leq 0$, the solution is $u^*=0$ and $k^* =- \frac{\lambda \alpha}{\delta}$, which is negative if $\alpha$ is positive. Thus we assume $\frac{p}{k^* +y^* +\epsilon} -c \geq 0$. Replacing $\bar{u} = u^*-\alpha$, we get 
\be 
\label{eq:38}
\begin{cases}
	\lambda \bar{u} -\delta k^*=0, &\\
	f(k^*, y^*)=0, &\\
	-(r+\delta) \bar{u} + \frac{hp}{k^* +y^* +\epsilon} -\bar{c} = 0. &
\end{cases}
\ee 
\\

We will consider a simple form of $f$ linear: with $a, b >0$
\be 
\label{eq:effe}
f(k,y) = -a k -b y +\gamma.
\ee
The constraints are satisfied if $\frac{\gamma}{a} \geq k_{max}$ and $\frac{-b y_{max}+\gamma}{a}\leq 0$. With this $f$ the above system provides the equation $y^* = \frac{-a k^* +\gamma}{b}$ and 
\be 
\delta k^2 + \left(\delta \frac{\frac{\gamma}{b}+\epsilon}{1-\frac{a}{b}} + \frac{\bar{c} \lambda}{r+\delta} \right) k   +\frac{\bar{c} \lambda }{r+\delta} \frac{\frac{\gamma}{b}+\epsilon}{1-\frac{a}{b}} 
- \frac{ \lambda}{r+\delta} \frac{hp}{1-\frac{a}{b}} =0,
\ee 
if $a\neq b$, which is equation \eqref{eq:2} with $\epsilon$ replaced by $\frac{-\frac{\gamma}{b}+\epsilon}{1-\frac{a}{b}}$ and $p$ replaced by  $\frac{p}{1-\frac{a}{b}}$. Hence the solution is 
\be 
k^*= \frac{ -\delta\frac{+\frac{\gamma}{b}+\epsilon}{1-\frac{a}{b}}-\frac{\bar{c} \lambda}{r+\delta}+ \sqrt{\left(\delta \frac{\frac{\gamma}{b}+\epsilon}{1-\frac{a}{b}}  - \frac{\bar{c} \lambda}{r+\delta}\right)^2 +4 \frac{\delta \lambda }{r+\delta}\frac{hp}{1-\frac{a}{b}}}}{2\delta} .
\ee 
and exists and is positive if 
\[
\frac{\bar{c}  \left(\frac{\gamma}{b}+\epsilon\right) -hp}{1-\frac{a}{b}} <0.
\]
If $a=b$ the solution is instead $k^*= 
\frac{\lambda}{\delta(r+\delta)} \left(\frac{hp}{\frac{\gamma}{a} +\epsilon} -\bar{c} \right)$. 


\subsection{Monopoly} 

We consider here the case where the production of the renewable energy is subject to a monopoly regime. In analogy to \eqref{master2}, the master equation with the reserve becomes 
\be 
\label{master4}
-(r+\delta)U + \left(\lambda( U - \alpha) -\delta k\right) \partial_k U + f(k,y) \partial_y U  +  \frac{hp (y+ \epsilon)}{(k+y+\epsilon)^2} -hc =0.
\ee
The equation is derived formally as in \S \ref{subsec:const_sub} , by considering  in addition the reserve $Y_t$ which is treated as a given function of time, in the optimization problem for the unique producer of the renewable energy. Roughly speaking, $\epsilon$ is replaced by $y+\epsilon$. 


In this case, the stationary state is given by a triple $(k^*, y^*, u^*)$. Similarly to the competitive case \eqref{eq:37}, it is provided by the system 
\be 
\label{eq:43}
\begin{cases}
	\lambda( u^* -\alpha ) -\delta k^*=0, &\\
	f(k^*, y^*)=0, &\\
	-(r+\delta) u^* + \frac{hp (y^*+ \epsilon)}{(k^*+y^*+\epsilon)^2} -hc = 0. &
\end{cases}
\ee 
In the case where $f$ is given by \eqref{eq:effe}, we obtain that $k^*$ is a root of the polynomial  
\begin{equation*} 
	\delta k^3 + \left( \frac{2(\epsilon +\!\frac{\gamma}{b})\delta}{1-\frac{a}{b}} + \frac{\bar{c} \lambda}{r\!+\!\delta}  \right) k^2 
	+ \left( \frac{\delta (\epsilon+\!\frac{\gamma}{b}) ^2}{\big(1-\frac{a}{b}\big)^2} + \frac{2(\epsilon+\!\frac{\gamma}{b} )\bar{c}\lambda}{(r\!+\!\delta) \big(1-\frac{a}{b}\big) } + \frac{a}{b} hp \frac{\lambda}{(r\!+\!\delta)\big(1-\frac{a}{b}\big)^2} \right) k 
	+ \frac{\lambda(\epsilon +\!\frac{\gamma}{b})\left( \bar{c} (\epsilon +\!\frac{\gamma}{b}) - \!h p    \right)}{(r+\delta)\big(1-\frac{a}{b}\big)^2} .
\end{equation*}

Similarly to Lemma \ref{lem:4}, we have the following:
\begin{lem}	
	For every admissible value of the parameters $\alpha, c$ and every choice of $f$ invertible, we have 
	\be 
	k^*_{mono} < k^*_{agents} ,
	\ee 
	where $k^*_{mono}$ and $k^*_{agents}$ are the equilibrium capacities in the monopoly regime and in the competitive small agents regime, respectively. 
\end{lem}

\begin{proof} 
	As in Lemma \ref{lem:4},  $u^*$ is a strictly increasing function of $k^*$, which is the same in the two regimes. In addition, we get $y^*$ is a function of $k^*$, say $y^*=\phi(k^*)$, since $f$ is invertible and does not depend on $u$.  Thus we have just to compare the third equation in the systems: since 
	\[
	\frac{p(\phi(k)+\epsilon)}{(k+ \phi(k) +\epsilon)^2} < \frac{p}{k+\phi(k) + \epsilon} \qquad 
	\forall k>0,
	\] 
	we obtain the claim. 
\end{proof}






\subsection{Numerical simulations} 
\label{sec:numerics:2}
To simulate the master equation \eqref{master:Y} we propose a finite difference scheme which takes into account the monotonicity in $k$ of $U$. 
The master equation \eqref{master:Y} is a scalar PDE in two space variables. The problem is that it is non linear and has no boundary conditions, thus usual finite difference methods are not directly applicable; see e.g. \cite{AchdouLauriere} for an overview of numerical methods for mean field games. To simulate it, we adapt the upwind scheme for linear transport equations, exploiting the monotonicity in $k$ of $U$, which ensures that the scheme is monotone. We hence simulate the equation without boundary conditions, using the monotonicity of the coefficients which ensure that there is an invariant domain where the equation can thus be numerically computed.

More precisely, We consider a uniform grid $(k_i, y_j)_{1\leq i\leq n, 1\leq j\leq m}$ which discretizes $[0, x_{max}]\times [0, y_{max}]$ with steps $\Delta k$ and $\Delta y$, for suitable maximum points $x_{max} $ and $y_{max}$, and consider the approximation 
\[
u^{n,m}_{i,j} \approx U(x_i,y_j).
\]
For $n$ and $m$ large enough, $u = (u_{i,j})_{1\leq i\leq n, 1\leq j\leq m}$ is a solution of the equation 
\be 
\begin{split}
	0 &=-(r+\delta) u_{i,j} 
	+ (\lambda( u_{i,j} -\alpha) -\delta k_i)_+
	\frac{u_{i+1,j}-u_{i,j}}{\Delta k}  
	- (\lambda( u_{i,j} -\alpha) -\delta k_i)_-
	\frac{u_{i,j}-u_{i-1,j}}{\Delta k} \\ 
	& + f_+(k_i, y_j) \frac{u_{i,j+1} - u_{i,j}}{\Delta y} 
	- f_-(k_i, y_j) \frac{u_{i,j} - u_{i,j-1}}{\Delta y}
	+ h\bigg( \frac{p}{k_i+y_j +\epsilon} -c\bigg)_+ ,
\end{split} 
\ee
for $0<i<n$, $0<j<m$, while the boundary points are such that (choosing properly $x_{max} $ and $y_{max}$) $\lambda( U -\alpha) -\delta k_i \leq 0$ in $x_{max}$ and $\geq 0$ in $0$, and 
$f(k, y_{max}) \leq 0$, $f(k, 0) \geq 0$. Splitting positive and negative part in this way is typical of upwind shemes for the transport equation.
This equation is in the form $F_{i,j}(u) =0$ and  is solved by Newton method, which consists of the iteration 
\[
u^{N+1} = u^N - (DF)^{-1} F(u^N).  
\]
This is solvable because $DF$ is invertible since it is diagonally strictly dominant (even if not properly defined) due to the fact that $U$ is decreasing in $k$. 
Indeed, we have 
\begin{align*}
	\frac{\partial F_{i,j}}{\partial u_{i,j}} 
	&= -(r+\delta) - | \lambda( u_{i,j} -\alpha) -\delta k_i | \frac{1}{\Delta k} - |f(k_i, y_j)| \frac{1}{\Delta y} \\
	& \quad+ \lambda \mathbbm{1}(\lambda( u_{i,j} -\alpha) -\delta k_i >0)\frac{u_{i+1,j}-u_{i,j}}{\Delta k} 
	+ \lambda \mathbbm{1}(\lambda( u_{i,j} -\alpha) -\delta k_i <0) \frac{u_{i,j}-u_{i-1,j}}{\Delta k} \\
	\frac{\partial F_{i,j}}{\partial u_{i+1,j}} &=  \frac{1}{\Delta k}(\lambda( u_{i,j} -\alpha) -\delta k_i )_+, \qquad\qquad  
	\frac{\partial F_{i,j}}{\partial u_{i-1,j}} =  \frac{1}{\Delta k}(\lambda( u_{i,j} -\alpha) -\delta k_i )_-, \\
	\frac{\partial F_{i,j}}{\partial u_{i,j+1}} &= \frac{1}{\Delta y}f_+(k_i, y_j) , \qquad\qquad  
	\frac{\partial F_{i,j}}{\partial u_{i,j-1}} = \frac{1}{\Delta y}f_-(k_i, y_j).
\end{align*}
If $u$ is decreasing in $k$, all the terms in $\frac{\partial F_{i,j}}{\partial u_{i,j}}$ are negative and thus 
\begin{align*}
	\bigg| \frac{\partial F_{i,j}}{\partial u_{i,j}} \bigg| &- 
	\sum_{(k,l)\neq (i,j)} \bigg| \frac{\partial F_{i,j}}{\partial u_{k,l}} \bigg|  = (r+\delta) \\
	& - \lambda \mathbbm{1}(\lambda( u_{i,j} -\alpha) -\delta k_i >0)\frac{u_{i+1,j}-u_{i,j}}{\Delta k} 
	- \lambda \mathbbm{1}(\lambda( u_{i,j} -\alpha) -\delta k_i <0) \frac{u_{i,j}-u_{i-1,j}}{\Delta k} >0.
\end{align*} 

For the numerical simulations, the idea is to avoid the constraints on $k=k_{max}$ and $y=0$, by simulating the master equation only in the rectangle $[k_0 , k^*]\times [y^*, y_0]$, which turns out (a posteriori) to be invariant,  where the initial conditions are $k_0 < k^*$ and $y_0 > y^*$.	\\

\textbf{Case} $a=b=1$.
In this case, it means that the reserve and the renewable capacity are perfect substitutes, and that when one unit of renewable is built, the reserve is going to decrease by one unit. And then $\gamma$ is the level of capacity (reserve + renewable) in the system at the equilibrium,
\be 
f(k, y) = - (k+y) +\gamma, 
\ee 
so that $k^*$ is given by the formula 
\be 
k^*= 
\frac{\lambda}{\delta(r+\delta)} \left(\frac{hp}{D +\epsilon} -(\bar{c} - \bar{c}_{sub}) \right) , \qquad y^* = \gamma- k^*.
\ee
by considering that the the renewable are subsidized by $\bar{c}_{sub}$.
We stress  that we have always the constraint $y_0 \leq \frac{p}{\bar{c}}-\epsilon$ if $\alpha >0$. 

We use the same parameters as in \S \ref{sec:constantsub_optimizesub} :
\begin{align*}
	r=0.1 /\mbox{year}, \qquad \delta=\ln(2)/10 \mbox{year},  \qquad \lambda=5 \frac{\text{MW}^2}{\mbox{\euro  year}^2} , \qquad \epsilon=0.1 \text{ MW}, \qquad p = 6.5*10^6 \mbox{ \euro}
\end{align*}
To achieve a target optimal capacity of renewable, the level of subvention should be:    
\[
\bar{c}- \bar{c}_{sub} = \frac{hp}{D+\epsilon} - k^* \frac{\delta (r+\delta)}{\lambda}. 
\]
%
%
%
%
%
%
%
%
%
%
%
We choose
\[
k_0 = 30 \text{ GW}, \qquad Y_0 = 100 \text{ GW}, 
\qquad k^* = 60 \text{ GW}, \qquad y^* = 70 \text{ GW}
\] 
and make the choice $a=b=1$, $\gamma = 130$ GW, thus we get that the optimal $\bar{c}- \bar{c}_{sub} $ is 150,000 \euro/MW. As $\bar{c}=282,040$, it means that the level of subvention is approximately half of the total cost of the renewable to guarantee to reach the desired equilibrium. \\

The following figures \ref{figure:5} and \ref{figure:6} are the graphics of  the solution to the master equation (in $[k_0, k^*]\times [y^*, y_0]$), 
the dynamics of $K_t$ and $Y_t$, the curve $(K_t, Y_t)$ parameterized by time in the space $(k,y)$, and the quantity $K_t+Y_t$.

\begin{figure}
	\centering
	\begin{subfigure}{.5\textwidth}
		\centering
		\includegraphics[scale=0.3]{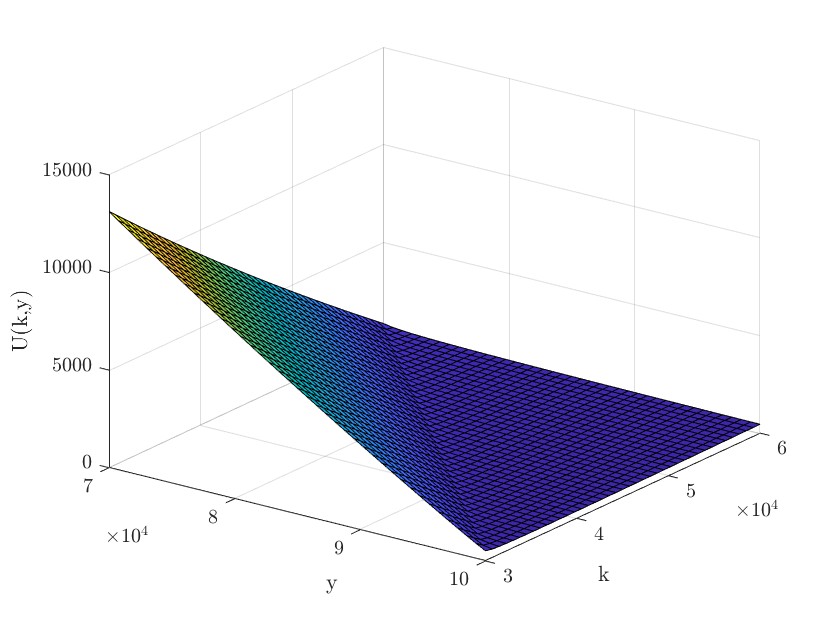}
	\end{subfigure}%
	\begin{subfigure}{.5\textwidth}
		\centering
		\includegraphics[scale=0.3]{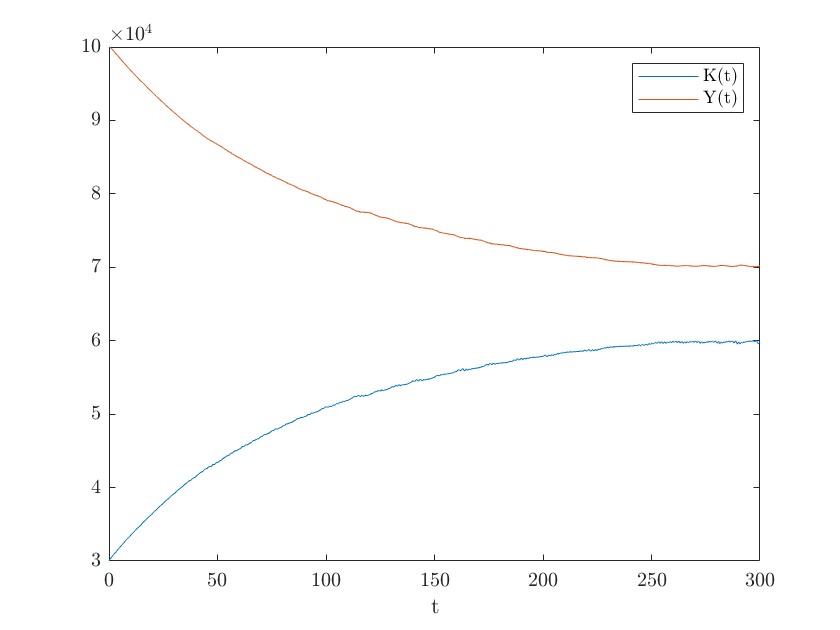}
	\end{subfigure}
	
	\caption{Value function U with respect of the renewable and reserve capacities (left) and time evolution of $K_t$ and $Y_t$ in MW (right)}
	\label{figure:5}
\end{figure}

\begin{figure}
	\centering
	\begin{subfigure}{.5\textwidth}
		\centering
		\includegraphics[scale=0.3]{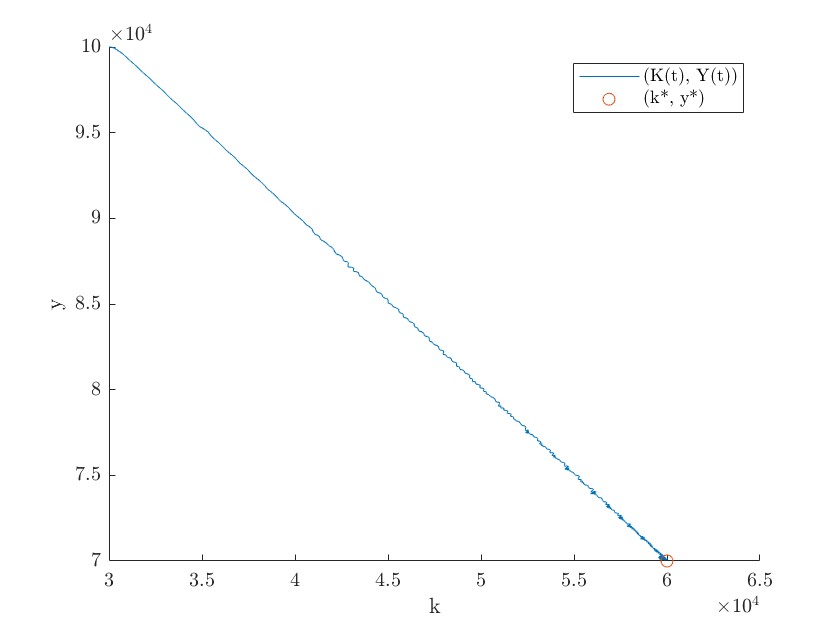}
	\end{subfigure}%
	\begin{subfigure}{.5\textwidth}
		\centering
		\includegraphics[scale=0.3]{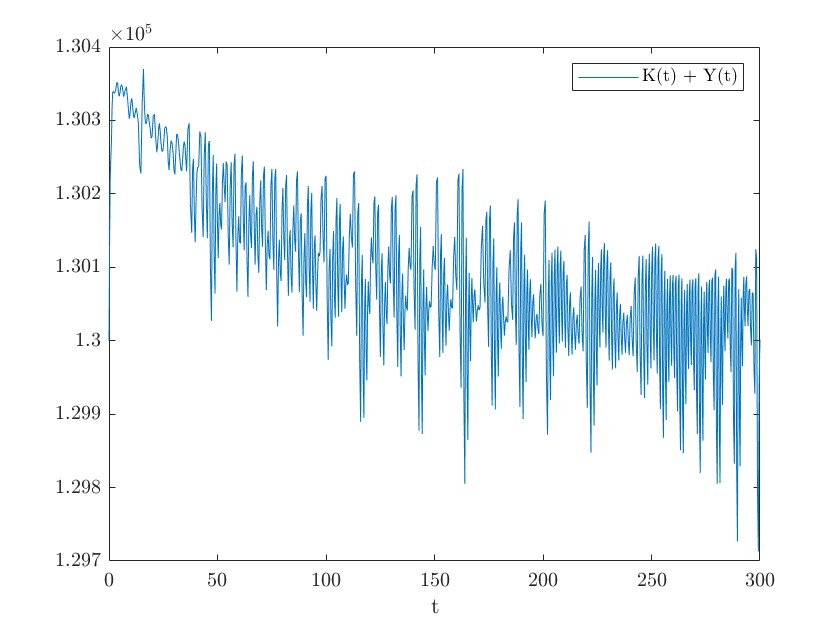}
	\end{subfigure}
	
	\caption{Dynamic of $Y_t$ and $K_t$ in MW (left) and time evolution of total capacities $K_t+Y_t$ in MW (right)}
	\label{figure:6}
\end{figure}

We see that $U$ is decreasing in $k$, there is convergence to equilibrium (quite slowly, though, $T=300$ years in the figures), and the sum of the two capacities stays above the demand. 
The numerical error at the terminal time is given by  
$K_{300}-k^* = -194$ MW,  
$Y_{300} - y^* = 195$ MW.

\vspace{0.5cm}

\textbf{Case} $a=2/3$, $b=1$.
In this case, the reserve and the renewable capacity are not perfect substitutes. When one unit of renewable is built, the reserve is only decreasing by one third of a unit in order to cope with variability of renewable production.

With the choice $a=2/3$, $b=1$, $\gamma = 130$ GW, we get the optimal $\bar{c}- \bar{c}_{sub} = 129,900$ \euro/MW and $y*=90$ GW. In this case, the level of subvention to provide to the new renewable technology is higher than in the previous case because the spot price is lower due to the higher level of reserve. We have convergence to equilibrium in a much shorter time, $T=15$ years in the figures, and the sum of the two capacities increases over time until the equilibrium value of 150 GW.  
The numerical error at the terminal time is  
$K_{50}-k^* = -28$ MW,  
$Y_{50} - y^* = 21$ MW, 
and we have the following figures.


\begin{figure}
	
	\centering
	\begin{subfigure}{.5\textwidth}
		\centering
		\includegraphics[scale=0.3]{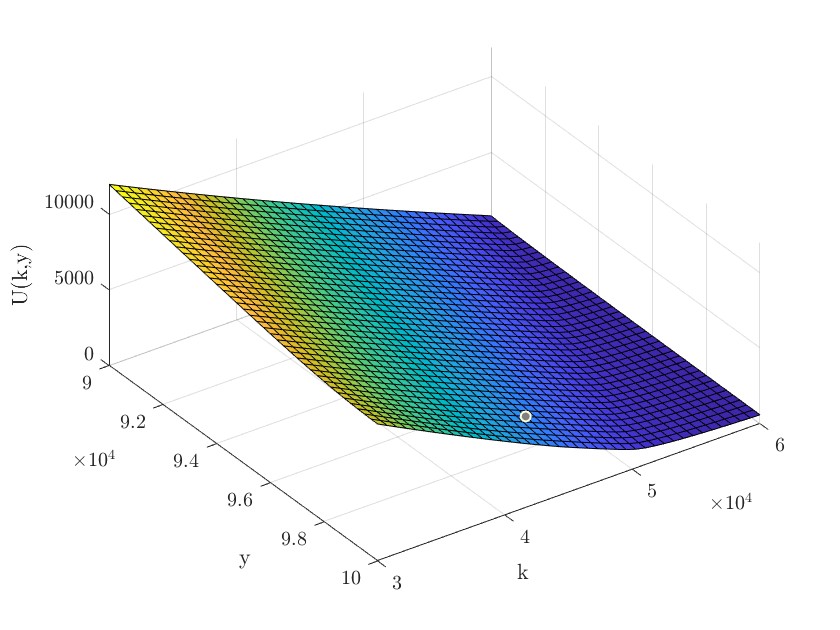}
	\end{subfigure}%
	\begin{subfigure}{.5\textwidth}
		\centering
		\includegraphics[scale=0.3]{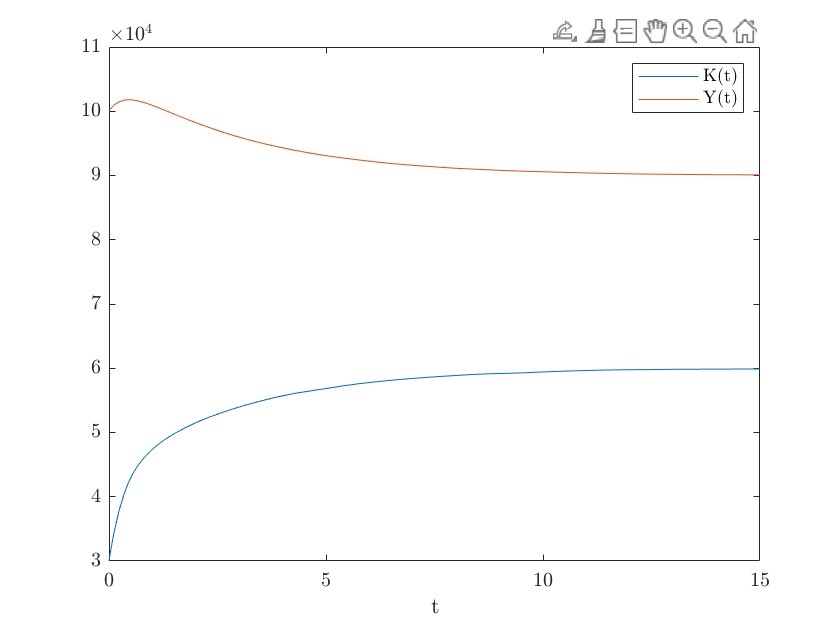}
	\end{subfigure}
	
	\caption{Value function U with respect of the renewable and reserve capacities (left) and time evolution of $K_t$ and $Y_t$ in MW (right)}
\end{figure}

\begin{figure}
	
	\centering
	\begin{subfigure}{.5\textwidth}
		\centering
		\includegraphics[scale=0.3]{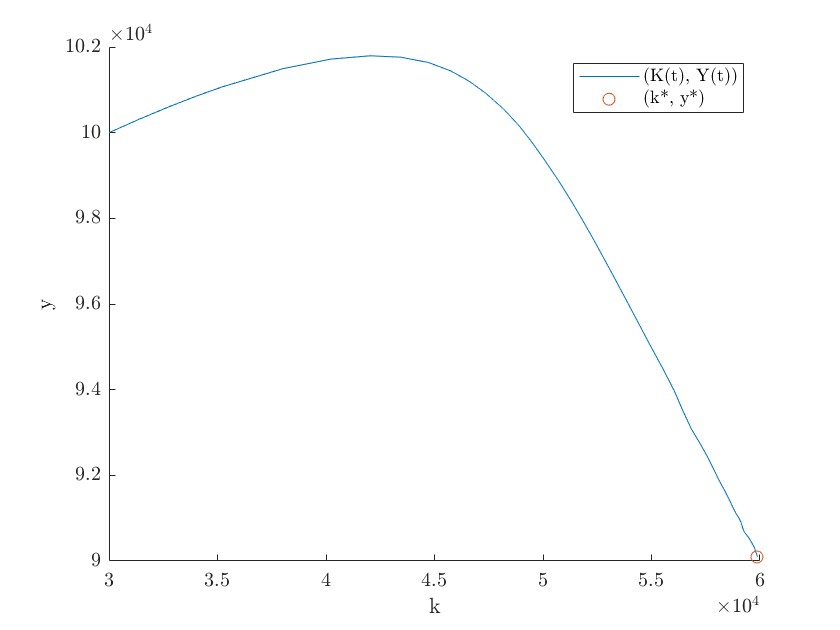}
	\end{subfigure}%
	\begin{subfigure}{.5\textwidth}
		\centering
		\includegraphics[scale=0.3]{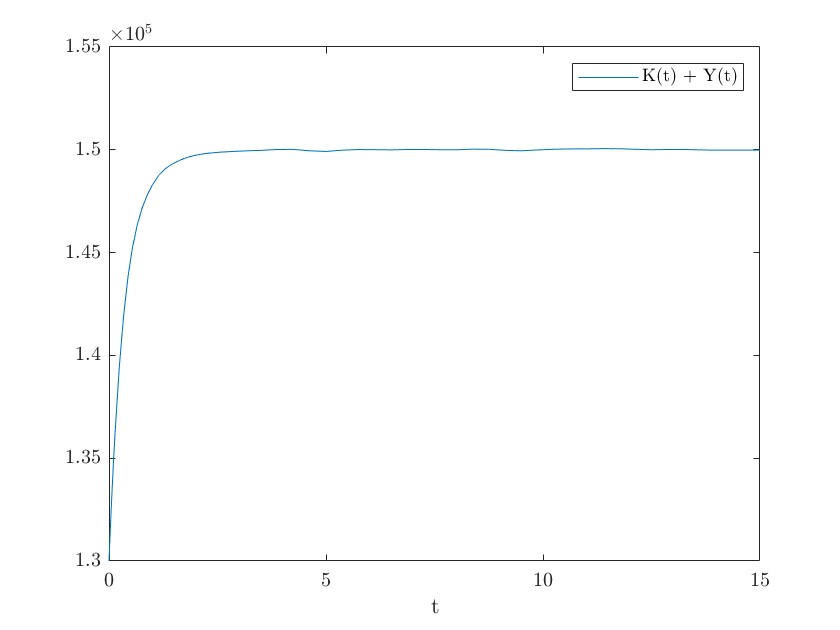}
	\end{subfigure}
	
	\caption{Dynamic of $Y_t$ and $K_t$ in MW (left) and time evolution of total capacities $K_t+Y_t$ in MW (right)}
\end{figure}







\end{document}